\documentclass[preprint,12pt]{elsarticle}
	
\usepackage{latexsym}
\usepackage{amsmath, amsthm, amsfonts, amssymb}
\usepackage{epsfig}	
\usepackage{hyperref}	

\newtheorem{thr}{Theorem}

\newtheorem{cor}{Corollary}
\newtheorem{lem}{Lemma}
 \newcommand{\F}{\mathbb{F}}

\journal{Finite Fields and Their Applications}

\begin{document}

\begin{frontmatter}



\title{The Number of Irreducible Polynomials over Finite Fields of Characteristic 2 with Given Trace and Subtrace}


\author[kisu]{Won-Ho Ri\corref{cor}}
\ead{riwonho@yahoo.com}
\cortext[cor]{Corresponding author}

\author[kisu]{Gum-Chol Myong}

\author[kisu]{Ryul Kim}

\author[kisu]{Chang-Il Rim}

\address[kisu]{Faculty of Mathematics, \textbf{Kim Il Sung} University, Pyongyang, D.P.R Korea}

 %
\begin{abstract}
In this paper we obtained the formula for the number of irreducible polynomials with degree $n$ over 
finite fields of characteristic two with given trace and subtrace. This formula is a generalization of the 
result of Cattell et al.\ (2003) \cite{cat}.
\end{abstract}


\end{frontmatter}



%
%
%
\section{Introduction}

In the theory of polynomials over finite fields the existence and the number of irreducible polynomials with some 
given coefficients have been investigated extensively. 
Hansen-Mullen conjecture states that for $n\geq3$, there exist irreducible polynomials of degree $n$ over a finite field $\textnormal{GF}(q)$ 
with any one coefficient given to any element of $\textnormal{GF}(q)$. 
This conjecture has already been settled completely and generalized to several classes of polynomials over finite fields. 
See, for example, \cite{ga2,pan,pol}.

In addition to the existence problem, calculating or estimating the number of irreducible polynomials over 
finite fields with some given coefficients have been studied by many researchers.
It is well known that a formula
\begin{equation*}
P(n)=\frac{1}{n} \sum_{\substack{d|n \\ d~ \text{odd}}} \mu (d)q^{n/d}
\end{equation*}
gives the number of monic irreducible polynomials of degree $n$ over $\textnormal{GF}(q)$, where  $\mu (d)$ is the 
M\"obius function \cite{lid}. Less well known is the formula
\begin{equation*}
P_1(n)=\frac{1}{qn} \sum_{\substack{d|n \\ d~ \text{odd}}} \mu (d)q^{n/d}
\end{equation*}
which counts the number of monic irreducible polynomials of degree $n$ over $\textnormal{GF}(q)$ that have a 
given nonzero trace \cite{car,cat,rus, yu2}. 

Cattell et al.\  \cite{cat} refined these formulas by enumerating the irreducible polynomials of degree $n$ 
over GF(2) with given trace and subtrace. 
The \textit{trace} of a monic irreducible polynomial $p(x)$ of degree $n$ over $\textnormal{GF}(q)$ is the coefficient of $x^{n-1}$ 
and the \textit{subtrace} is the coefficient of $x^{n-2}$. 
The result obtained in \cite{cat} is that the number of degree $n$ irreducible polynomials over GF(2) with given trace and subtrace 
is covered by one of the following cases:\\
\indent $\bullet$ The number of trace 0, subtrace 0 polynomials is $\sum_{k \equiv 2n+2 \pmod{4}}L(n, k)$\\
\indent $\bullet$ The number of trace 0, subtrace 1 polynomials is $\sum_{k \equiv 2n \pmod{4}}L(n, k)$\\
\indent $\bullet$ The number of trace 1, subtrace 0 polynomials is $\sum_{k \equiv 2n-1 \pmod{4}}L(n, k)$\\
\indent $\bullet$ The number of trace 1, subtrace 1 polynomials is $\sum_{k \equiv 2n+1 \pmod{4}}L(n, k)$\\
\noindent Here $L(n, k)$ is the number of binary Lyndon words of length $n$ containing exactly $k$ 1's.
A binary Lyndon word of length n is an $n$-character string over an alphabet of size 2 (e.g, 0 and 1),
and which is the minimum element in the lexicographical ordering of all its rotations.
It is known that
\begin{equation*}
L(n, k)=\frac{1}{n} \sum_{d|\text{gcd}(n, k)} \mu (d) \binom{n/d}{k/d}.
\end{equation*}
\noindent In \cite{cat}, a generalized M\"obius Inversion Formula was proved and used (\cite{cat}, Theorem 1). 

Yucas and Mullen \cite{yu1} enumerated the number of irreducible polynomials of even degree over GF(2) with the first three coefficients prescribed. 
Fitzgerald and Yucas \cite{fit} counted the number of irreducible polynomials of odd degree over GF(2) with the first three coefficients prescribed. 
Niederreiter \cite{nie} obtained the formula for the number of irreducible polynomials over GF(2) with given trace and cotrace 
to apply to the construction of irreducible polynomials over a binary field. We refer to \cite{ga1,kom,mo1,mo2} 
for more results on this line.

The purpose of this paper is to generalize the results obtained in \cite{cat} by enumerating the irreducible polynomials of degree 
$n$ over any finite field of characteristic 2 with given trace and subtrace.

In this paper, $q=2^k$ and for a degree $n$ polynomial $p(x)$ over $\textnormal{GF}(q)$, we denote the trace and subtrace of $p(x)$ by $Tr(p)$ and $St(p)$, respectively. 
And the trace and subtrace of an element $\beta$ in $\textnormal{GF}(q^n)$ are defined to be
\begin{equation*}
Tr(\beta)=\sum_{i=0}^{n-1} \beta ^{q^i} \textnormal{ and } \quad  St(\beta)=\sum_{0\leq i<j<n} \beta ^{q^i}\beta ^{q^j},
\end{equation*}
respectively. If $t, s \in \textnormal{GF}(q)$, then let $P(n, t, s)$ be the number of monic irreducible polynomials of degree $n$ over 
$\textnormal{GF}(q)$ of which trace and subtrace are $t, s$, respectively, and $F(n, t, s)$ be the number of elements in $\textnormal{GF}(q^n)$ of which trace and subtrace are $t, s$, respectively.

We use the generalized M\"obius Inversion Formula to show the relationship between $P(n, t, s)$ and $F(n, t, s)$ and then count $F(n, t, s)$ as in \cite{cat}.
Congruence modulo 4 are pervasive in this paper and expressions of the form $x \equiv y \pmod{4}$  are abbreviated as $x \equiv y$. 
For a proposition $P$, [$P$] represents its truth value: [$P$] has value 1 if $P$ is true and value 0 if $P$ is false.
%
%
 %
%
\section{Main Results}

The following theorem can be proved using the generalized M\"obius Inversion Formula in a similar way to that of \cite{cat}, 
so here we only give the sketch of its proof.

%
%

\begin{thr} \label{thr1}
Let $t \in \textnormal{GF}(q)^*, s \in \textnormal{GF}(q)$. Then the following hold true.
\begin{eqnarray*}
nP(n, 0, s) & = & \sum_{\substack{d|n \\ d~ \textnormal{odd}}} \mu (d) \Big( F(n/d, 0, s)-[n~ \textnormal{even}] q^{n/{(2d)}-1}\Big). \\
nP(n, t, s) & = & \sum_{\substack{d|n \\ d \equiv 1}} \mu (d) F(n/d, t, s)+\sum_{\substack{d|n \\ d \equiv 3}} \mu (d) F(n/d, t, t^2+s).
\end{eqnarray*}
\end{thr}
\begin{proof}[Sketch of proof]
If $\beta\in\textnormal{GF}(q^n)$ and $p(x)$ is the minimal polynomial of $\beta$, denoted by $Min(\beta)$, then
\begin{equation*}
p(x)=(x-\beta)(x-\beta^q)\cdots(x-\beta^{q^{d-1}}),
\end{equation*}
where $d|n$. Let Irr($n$) denote the set of all irreducible polynomials over GF(q) of degree $n$.
By $a\cdot \textnormal{Irr}(n)$ we denote the multiset consisting of $a$ copies of $\textnormal{Irr}(n)$.
Classic results of finite field theory imply the following equality of multisets:
\begin{equation*}
\bigcup_{\beta\in\textnormal{GF}(q^n)}{Min(\beta)}=
\bigcup_{d|n}{d\cdot\textnormal{Irr}(d)}=
\bigcup_{d|n}{\frac{n}{d}\cdot\textnormal{Irr}\left(\frac{n}{d}\right)}.
\end{equation*}
From this, it is easy to see that the following holds true.
\begin{eqnarray*}
&& \bigcup_{\substack{\beta\in\textnormal{GF}(q^n)\\ Tr(\beta)=t \\ St(\beta)=s}}
Min(\beta) =\bigcup_{d|n}\frac{n}{d}\left\{p\in\textnormal{Irr}\left(\frac{n}{d}\right):Tr(p^d)=t, ~ St(p^d)=s\right\} \\
&& \qquad \qquad = \bigcup_{d|n}\frac{n}{d}\left\{p\in\textnormal{Irr}\left(\frac{n}{d}\right):d\cdot Tr(p)=t, ~ d\cdot St(p)+\binom{d}{2}Tr(p)=s \right\} \\
&& \qquad \qquad = \bigcup_{\substack{d|n \\ d\equiv 0}} \frac{n}{d}\left\{p\in\textnormal{Irr}\left(\frac{n}{d}\right):t=0, ~ s=0\right\} \\ && \qquad \qquad \qquad \qquad \cup \bigcup_{\substack{d|n \\ d\equiv 1}} \frac{n}{d}\left\{p\in\textnormal{Irr}\left(\frac{n}{d}\right):Tr(p)=t, ~ St(p)=s\right\} \\
&& \qquad \qquad \qquad \qquad \cup \bigcup_{\substack{d|n \\ d\equiv 2}} \frac{n}{d}\left\{p\in\textnormal{Irr}\left(\frac{n}{d}\right):t=0, ~ Tr(p)=s\right\} \\
&& \qquad \qquad \qquad \qquad \cup \bigcup_{\substack{d|n \\ d\equiv 3}} \frac{n}{d}\left\{p\in\textnormal{Irr}\left(\frac{n}{d}\right):Tr(p)=t, ~ St(p)=s+t\right\}
\end{eqnarray*}

Taking cardinalities, we obtain an expression for $F(n, t, s)$ in terms of $P(n, t, s)$.
\begin{eqnarray*}
&&F(n, t, s) = \sum_{\substack{\beta\in\textnormal{GF}(q^n) \\Tr(\beta)=t \\St(\beta)=s}}|Min(\beta)| \\
&& \quad =  \sum_{\substack{d|n \\ d\equiv 0}}\frac{n}{d}\cdot \left\vert \left\{ p\in\textnormal{Irr}\left(\frac{n}{d}\right): t=0, ~ s=0\right\}\right\vert  + \sum_{\substack{d|n \\ d\equiv 1}}\frac{n}{d}\cdot F(n/d, t, s) \\
&& \quad + \sum_{\substack{d|n \\ d\equiv 2}}\frac{n}{d}\cdot \left\vert\left\{ p\in\textnormal{Irr}\left(\frac{n}{d}\right):t=0, ~ Tr(p)=s\right\} \right\vert +\sum_{\substack{d|n \\ d\equiv 3}}\frac{n}{d}\cdot F(n/d, t, s+t).
\end{eqnarray*}

Applying the generalized Moebius Inversion Theorem (\cite{cat}, Therem 1), we get the result of the theorem. 
\end{proof}

Theorem \ref{thr1} means that in order to calculate $P(n, t, s)$ it suffices to calculate $F(n, t, s)$. The value of $F(n, t, s)$ is differently calculated according to $n$.\\
\indent {\bf Case 1 : } $n$ is odd. \\
\indent {\bf Case 2 : } $n \equiv 2 \pmod{4}$. \\
\indent {\bf Case 3 : } $n \equiv 0 \pmod{4}$. \\
\indent The first two cases use the existence of a self-dual normal basis. A self-dual normal basis exists for $\textnormal{GF}(q^n)$ over $\textnormal{GF}(q)$ if and only if $n$ is not a multiple of 4.
In the following three theorems, for $s \in \textnormal{GF}(q), v(s)=q-1$ if $s=0$, and $v(s)=-1$ if $s \neq 0$. (\cite{lid}, (6.22))

%
\begin{thr} \label{thr2}
Let $n \geq 2$ be odd and $t, s \in \textnormal{GF}(q)$. Then\\
\indent \textnormal{1)} If $n=4m+1$, then $F(n, t, s)=q^{n-2}+(-1)^{mk}v(s)q^{2m-1}$\\
\indent \textnormal{2)} If $n=4m-1$, then $F(n, t, s)=q^{n-2}+(-1)^{(m-1)k}v(t^2-s)q^{2m-2}$\\	
\end{thr}	
%
%
\begin{thr} \label{thr3}
Let $n=4m+2$ and $t, s \in \textnormal{GF}(q)$. Then\\
\begin{equation*}
F(n, t, s)=\left \{ 
\begin{array}{ll}
	q^{n-2}, & t=0,\\
	q^{n-2}-(-1)^{mk}\chi (s/t^2)q^{2m}, & t \neq 0,
\end{array} \right.
\end{equation*}
where $\chi$ is the canonical character on $\textnormal{GF}(q)$, i.e, 
$\chi(t) = (-1)^{Tr_{q:2}(t)}, t\in \textnormal{GF}(q)$.
$Tr_{q:2}(t)$ is the trace from $\textnormal{GF}(q)$ to $\textnormal{GF}(2)$, that is, $Tr_{q:2}(t)=t+t^2+\cdots+t^{2^{k-1}} \in \textnormal{GF}(2)$.
\end{thr}	
%
%
\begin{thr} \label{thr4}
Let $n=4m$ and $t, s \in \textnormal{GF}(q)$. Then\\
\begin{equation*}
F(n, t, s)=\left \{ 
\begin{array}{ll}
	q^{n-2}-(-1)^{mk}v(s)q^{2m-1}, & t=0,\\
	q^{n-2}, & t \neq 0.
\end{array} \right.
\end{equation*}
\end{thr}	

%
%
 %
%
\section{Proofs}
Lemma \ref{lem1} and Lemma \ref{lem2} below are very similar to Lemma 8 and Lemma 9 of \cite{cat}, respectively, and their proofs are omitted.

%
%
	
\begin{lem} \label{lem1}
Let  $\beta \in \textnormal{GF}(q^n)$.
If $n=2m$, then $St(\beta)=Tr(\beta ^{q+1})+Tr(\beta ^{q^2+1})+ \cdots +Tr(\beta ^{q^{m-1}+1})+Tr_{q^m}(\beta ^{q^m+1})$.
If $n=2m+1$, then $St(\beta)=Tr(\beta ^{q+1})+Tr(\beta ^{q^2+1})+ \cdots +Tr(\beta ^{q^{m-1}+1})$.
\end{lem}
	
%
%
\begin{lem} \label{lem2}
Let \{$\theta, \theta^q, \cdots, \theta^{q^{n-1}}$\} be a self-dual normal basis for $\textnormal{GF}(q^n)$ over $\textnormal{GF}(q)$. 
Let $\beta=a_0\theta+a_1\theta^q+\cdots+a_{n-1}\theta^{q^{n-1}}$. Then $Tr(\beta)=\sum_{0\leq i<n}a_i$.
\end{lem}
Now we use the following notation. If $n\geq 2$ and $t, s \in \textnormal{GF}(q)$, then let
\begin{equation*}
F^*(n, t, s)=\Big \arrowvert \Big\{ (a_0, a_1, \cdots, a_{n-1}) \in \textnormal{GF}(q)^n : \sum_{0 \leq i<n}a_i=t, 
\sum_{0 \leq i<j<n}a_ia_j=s \Big\} \Big\arrowvert.
\end{equation*}

%
%

\begin{thr} \label{thr5}
Let  $n\geq 2$ and $t, s \in$ \textnormal{GF}($q$). Then\\
\indent \textnormal{1)} If $n=4m+1$, then $F^*(n, t, s)=q^{n-2}+(-1)^{mk}v(s)q^{2m-1}$ . \\
\indent \textnormal{2)} If $n=4m+2$, then $F^*(n, t, s)=\left \{ 
\begin{array}{ll}
	q^{n-2}, & t=0,\\
	q^{n-2}+(-1)^{mk}\chi (s/t^2)q^{2m}, & t \neq 0.
\end{array} \right.$ \\
\indent \textnormal{3)} If $n=4m-1$, then $F^*(n, t, s)=q^{n-2}+(-1)^{(m-1)k}v(t^2-s)q^{2m-2}$. \\
\indent \textnormal{4)} If $n=4m$, then $F^*(n, t, s)=\left \{ 
\begin{array}{ll}
	q^{n-2}+(-1)^{mk}v(s)q^{2m-1}, & t=0,\\
	q^{n-2}, & t\neq 0.
\end{array} \right.$ \\
\end{thr}
\begin{proof} 
We use induction on $n$.\\
\indent If $n=2$, since $F^*(2, t, s)=\big \arrowvert \big\{ (a, b) \in \textnormal{GF}(q)^2 : a+b=t, ab=s \big\} \big\arrowvert$, 
$F^*(2, t, s)$ is equal to the number of roots of quadratic polynomial $x^2+tx+s$ over $\textnormal{GF}(q)$. 
By 3.79 of \cite{lid}, the number is 1 if $t=0$, 2 if $t\neq 0, Tr_{q:2}(s/t^2)=0$, and 0 if $t\neq 0, Tr_{q:2}(s/t^2)=1$. 
Thus the theorem holds true in this case.
	
Suppose the theorem holds true in case of $n-1$ and consider the case of $n$. 
First, we derive a recursive relation on $F^*(n, t, s)$. Fix $a_{n-1}=\alpha$. Then 
\begin{equation*}
\sum_{0 \leq i<n}a_i=t, \quad \sum_{0 \leq i<j<n}a_ia_j=s 
\end{equation*}\\
\noindent if and only if
\begin{equation*}
\sum_{0 \leq i<n-1}a_i=t+\alpha, \quad \sum_{0 \leq i<j<n-1}a_ia_j=s+\alpha t+\alpha ^2.
\end{equation*}
Therefore we obtain
\begin{equation}
F^*(n, t, s)=\sum_{\alpha \in \textnormal{GF}(q)}F^*(n-1, t+\alpha, s+\alpha t+\alpha ^2). \label{eq1}
\end{equation}
Substituting $\beta = t+\alpha$, we have
\begin{align*}
F^*(n, t, s) & = \sum_{\beta \in \textnormal{GF}(q)}F^*(n-1, \beta, s+\beta t+\beta ^2) \\
		& = \sum_{\alpha \in \textnormal{GF}(q)}F^*(n-1, \alpha, s+\alpha t+\alpha ^2).
\end{align*}
	
1) ~ Case $n=4m+1:$ Since $n-1=4m$, from the induction hypothesis, we have
\begin{equation*}
F^*(n-1, \alpha, s+\alpha t+\alpha^2)=\left \{ 
\begin{array}{ll}
	q^{n-3}+(-1)^{mk}v(s)q^{2m-1}, & \alpha=0,\\
	q^{n-3}, & \alpha \neq 0.
\end{array} \right.\\
\end{equation*}
\noindent So by \eqref{eq1} we get
\begin{align*}
F^*(n, t, s) & =  [q^{n-3}+(-1)^{mk}v(s)q^{2m-1}]+(q-1)q^{n-3} \\
		  & =  q^{n-2}+(-1)^{mk}v(s)q^{2m-1}.
\end{align*}

2) ~ Case $n=4m+2:$ Suppose $t=0$. Since $n-1=4m+1$, from the induction hypothesis
\begin{equation*}
F^*(n-1, \alpha, s+\alpha^2)=q^{n-3}+(-1)^{mk}v(s+\alpha^2)q^{2m-1}.
\end{equation*}
\noindent So by \eqref{eq1}, we have
\begin{align*}
F^*(n, 0, s) & = q^{n-3}+(q-1)(-1)^{mk}q^{2m-1}+(q-1)(q^{n-3}-(-1)^{mk}q^{2m-1})\\
		& =q^{n-2}.
\end{align*}

Now suppose $t \neq 0$. From the induction hypothesis, we obtain
\begin{equation*}
F^*(n-1, \alpha, s+\alpha t+\alpha^2)=\left \{ 
\begin{array}{ll}
	q^{n-3}+(-1)^{mk}(q-1)q^{2m-1}, & s+\alpha t+\alpha^2=0,\\
	q^{n-3}-(-1)^{mk}q^{2m-1}, & s+\alpha t+\alpha^2 \neq 0.
\end{array} \right.\\
\end{equation*}
If $Tr_{q:2}(s/t^2)=0$, then there are two $\alpha$'s with $s+\alpha t+\alpha^2=0$  in $\textnormal{GF}(q)$ by 3.79 of \cite{lid}. 
Thus we have
\begin{align*}
F^*(n, t, s) & =  2[q^{n-3}+(-1)^{mk}(q-1)q^{2m-1}]+(q-2)[q^{n-3}-(-1)^{mk}q^{2m-1}] \\
		  & =  q^{n-2}+(-1)^{mk}q^{2m}.
\end{align*}
Similarly, if $Tr_{q:2}(s/t^2)=1$, then $F^*(n, t, s)=q^{n-2}-(-1)^{mk}q^{2m}$, as was to be shown.

3) ~ Case $n=4m-1:$ Since $n-1=4m-2$, from the induction hypothesis, we have
\begin{equation*}
F^*(n-1, \alpha, s+\alpha t+\alpha^2)=\left \{ 
\begin{array}{ll}
	q^{n-3}, & \alpha=0,\\
	q^{n-3}+(-1)^{(m-1)k}\chi \big( \frac{s+\alpha t+\alpha^2}{\alpha^2}\big) q^{2m-2}, & \alpha \neq 0.
\end{array} \right.\\
\end{equation*}
\noindent By virtue of properties of character, 
\begin{align*}
\chi \Big( \frac{s+\alpha t+\alpha^2}{\alpha^2}\Big) & =  \chi \Big( \frac{s}{\alpha^2}+ \frac{t}{\alpha}+1\Big)=
			(-1)^k \chi \Big( \frac{s}{\alpha^2}+ \frac{t}{\alpha}\Big) \\
		& = (-1)^k \chi (s/\alpha^2)\chi (t/\alpha)= (-1)^k \chi (\sqrt{s}/\alpha)\chi (t/\alpha)\\
		& = (-1)^k\chi ((\sqrt{s}+t)/\alpha),
\end{align*}
where $\sqrt{s}=s^{q/2}$.  Thus by \eqref{eq1} we get
\begin{align*}
F^*(n, t, s) & =  q^{n-3}+\sum_{\alpha \in \textnormal{GF}(q)^*}\big[ q^{n-3}+(-1)^{(m-1)k} \chi ((\sqrt{s}+t)/\alpha)q^{2m-2} \big]\\
		& =  q^{n-3}+\sum_{\alpha \in \textnormal{GF}(q)^*}\big[ q^{n-3}+(-1)^{(m-1)k} \chi_{\sqrt{s}+t} (1/\alpha)q^{2m-2} \big],
\end{align*}
where $\chi_t(s) = \chi(ts)$.

By properties of trace, we obtain
\begin{equation*}
\big \arrowvert \big\{ \alpha \in \textnormal{GF}(q)^*:\chi_{\sqrt{s}+t} (1/\alpha)=1\big\} \big \arrowvert = \left \{ 
\begin{array}{ll}
	q-1, & \sqrt{s}+t=0,\\
	q/2-1, & \sqrt{s}+t \neq 0.
\end{array} \right.\\
\end{equation*}
Therefore if $s=t^2$, then 
\begin{equation*}
F^*(n, t, s)=q^{n-2}+(-1)^{(m-1)k}(q-1)q^{2m-2},
\end{equation*}
and if $s \neq t^2$, then 
\begin{equation*}
F^*(n, t, s)=q^{n-2}-(-1)^{(m-1)k}q^{2m-2}.
\end{equation*}

4) ~ Case $n=4m:$ Since $n-1=4m-1$,
\begin{equation*}
F^*(n-1, \alpha, s+\alpha t+\alpha^2)=q^{n-3}+(-1)^{mk}v(s+\alpha t)q^{2m-2}.
\end{equation*}
So by \eqref{eq1} 
\begin{equation*}
F^*(n, t, s)=\sum_{\alpha \in \textnormal{GF}(q)} \big[ q^{n-3}+(-1)^{mk}v(s+\alpha t)q^{2m-2} \big].
\end{equation*}
If $t=0$, then $F^*(n, t, s)=q^{n-2}+(-1)^{mk}v(s)q^{2m-1}$. Now suppose $t \neq 0$. Since as $\alpha$ runs over all the 
elements of $\textnormal{GF}(q)$, so does $s+\alpha t$, we have
\begin{equation*}
F^*(n, t, s)=q^{n-2}+(-1)^{mk}q^{2m-2}\sum_{\alpha \in \textnormal{GF}(q)}v(\alpha)=q^{n-2}.
\end{equation*}
The proof completed. 
\end{proof} 

\noindent {\bf Case 1 : } $n$ is odd. \\
\indent We can easily see that the following analogue of Theorem 5 of \cite{cat} holds true and so we omit its proof.

%
%

\begin{lem} \label{lem3}
Let $n$ be odd. Let \{$\theta, \theta^q, \cdots, \theta^{q^{n-1}}$\} be a self-dual normal basis for $\textnormal{GF}(q^n)$ over $\textnormal{GF}(q)$. 
Let $\beta=a_0\theta+a_1\theta^q+\cdots+a_{n-1}\theta^{q^{n-1}} \in \textnormal{GF}(q^n)$. Then
\begin{equation*}
St(\beta)=\sum_{0 \leq i<j<n}a_ia_j.
\end{equation*}
\end{lem}
%
%
\begin{cor} \label{cor1}
If $n$ is odd, then $F(n, t, s)=F^*(n, t, s)$.
\end{cor}
Theorem \ref{thr5}, Lemma \ref{lem2} and Corollary \ref{cor1} yield Theorem \ref{thr2}. \\

\noindent {\bf Case 2 : } $n \equiv 2 \pmod{4}$.\\
\indent Let $n=4m+2$. We denote $Tr_{q^{2m+1}:q}$ by $Tr_{q^{2m+1}}$. The following lemma is analogous
to Lemma 11 of \cite{cat} and its proof is omitted.
%
%

\begin{lem} \label{lem4}
If $\theta \in \textnormal{GF}(q^n)$, then 
 \begin{equation*}
Tr_{q^n:q^2}(\theta) Tr_{q^n:q^2}(\theta^q)=\sum_{k=1}^{m}Tr(\theta^{q^{2k-1}+1})+Tr_{q^{2m+1}}(\theta^{q^{2m+1}+1}).
\end{equation*}
\end{lem}
%
%
\begin{cor} \label{cor2}
If $\theta \in \textnormal{GF}(q^n)$ generates a self-dual normal basis, then
\begin{equation*}
Tr_{q^n:q^2}(\theta) Tr_{q^n:q^2}(\theta^q)=Tr_{q^{2m+1}}(\theta^{q^{2m+1}+1}).
\end{equation*}
\end{cor}
%
%

\begin{lem} \label{lem5}
Assume that $\theta \in \textnormal{GF}(q^n)$ generates a self-dual normal basis. Let $\varepsilon = Tr_{q^{2m+1}}(\theta^{q^{2m+1}+1} )$, then $Tr_{q:2}(\varepsilon)=1$.
 \end{lem}
\begin{proof}
Let $\alpha=Tr_{q^{n}:q^2}(\theta)$. Then $\alpha^q=Tr_{q^{n}:q^2}(\theta^q)$  along with Corollary \ref{cor2} imply $\varepsilon=\alpha \alpha^q$.
Since $\alpha+\alpha^q=Tr_{q^2:q}[Tr_{q^{n}:q^2}(\theta)]=Tr(\theta)=1$, $\alpha$ is a root of quadratic polynomial $x^2+x+\varepsilon \in \textnormal{GF}(q)[x]$. 
Since $\alpha \neq \alpha^q$, i.e. $\alpha \notin \textnormal{GF}(q)$, this polynomial has no roots in $\textnormal{GF}(q)$  and so is irreducible. 
Thus we get $Tr_{q:2}(\varepsilon)=1$ from 3.79 of \cite{lid}. 
\end{proof}

The following theorem is similar to Theorem 6 of \cite{cat} and we omit its proof.
%
%

\begin{thr} \label{thr6}
Let $B=\{\theta, \theta^q, \cdots, \theta^{q^{n-1}}\}$ be a self-dual normal basis for $\textnormal{GF}(q^n)$ over $\textnormal{GF}(q)$. 
Let $\beta=a_0\theta+a_1\theta^q+\cdots+a_{n-1}\theta^{q^{n-1}} \in \textnormal{GF}(q^n)$. Then
\begin{equation*}
St(\beta)=\sum_{0 \leq i<j<n}a_ia_j+\varepsilon \Big( \sum_{0 \leq i<n}a_i \Big)^2,
\end{equation*}
where $\varepsilon = Tr_{q^{2m+1}}(\theta ^{q^{2m+1}+1})$.
\end{thr}
%
%
\begin{cor} \label{cor3}
If $n \equiv2$ and $t, s \in \textnormal{GF}(q)$ then $F(n, t, s)=F^*(n, t, s+\varepsilon t^2)$.
\end{cor}

Note that $Tr_{q:2}((s+\varepsilon t^2)/t^2)=Tr_{q:2}(s/t^2+\varepsilon)=Tr_{q:2}(s/t^2)+1$  by Lemma \ref{lem5}. This fact along with Theorem \ref{thr5} and Corollary \ref{cor3} imply Theorem \ref{thr3}.\\

\noindent {\bf Case 3: } $n \equiv 0 \pmod{4}$.\\

Let $n=4m$. If $\alpha \in \textnormal{GF}(q^{2m})$, let $R_{\alpha}=\big \{ \beta \in \textnormal{GF}(q^n): \beta ^{q^{2m}}+ \beta =\alpha \big\}$. 
$R_{\alpha}$  is the set of  $\beta \in \textnormal{GF}(q^n)$  with $Tr_{q^n:q^{2m}}(\beta)=\alpha$. $\textnormal{GF}(q^{2m})=R_0$ 
is a subspace of $\textnormal{GF}(q^n)$ and each $R_{\alpha}$  is a coset of $\textnormal{GF}(q^n)$  modulo $\textnormal{GF}(q^{2m})$. 
Thus  $\textnormal{GF}(q^n)$  is divided into  $q^{2m}$ cosets, each having  $q^{2m}$ elements. 
We count the elements with given trace and subtrace in each coset $R_{\alpha}$  and then add them up to calculate $F(n, t, s)$.

%
%

\begin{lem} \label{lem6}
Let $\alpha \in \textnormal{GF}(q^{2m})$. Then for any $\gamma \in R_{\alpha}, Tr(\gamma)=Tr_{q^{2m}}(\alpha)$.
\end{lem}
\begin{proof}
$Tr(\gamma)=Tr_{q^{2m}}(Tr_{q^n:q^{2m}}(\gamma))=Tr_{q^{2m}}(\gamma+\gamma^{q^{2m}})=Tr_{q^{2m}}(\alpha)$.
\end{proof}

Lemma \ref{lem7} below is similar to Lemma 14 of \cite{cat} and we omit its proof.
%
%
\begin{lem} \label{lem7}
Let $\alpha, \beta  \in \textnormal{GF}(q^{2m})$ and $\gamma \in R_{\alpha}$. Then
\begin{equation*}
St(\beta+\gamma)=St(\gamma)+[Tr_{q^{2m}}(\beta)]^2+Tr_{q^{2m}}(\alpha \beta).
\end{equation*}
\end{lem}
%
%
\begin{cor} \label{cor4}
If $\beta \in \textnormal{GF}(q^{2m})$ then $St(\beta)=[Tr_{q^{2m}}(\beta)]^2$.
\end{cor}

Now we introduce the following notation. If $\alpha \in \textnormal{GF}(q^{2m})$ and $t, s \in \textnormal{GF}(q)$, then 
let $F(n, t, s, \alpha)= \big \arrowvert \big\{ \gamma \in R_{\alpha}: Tr(\gamma)=t, St(\gamma)=s  \big\}\big \arrowvert$. Then
\begin{equation*}
F(n, t, s)=\sum_{\alpha \in \textnormal{GF}(q^{2m})}F(n, t, s, \alpha).
\end{equation*}
%
%

\begin{lem} \label{lem8}
If $\alpha \notin \textnormal{GF}(q)$, then
\begin{equation*}
F(n, t, s, \alpha)= \left \{ 
\begin{array}{ll}
	q^{2m-1}, & t=Tr_{q^{2m}}(\alpha)\\
	0, & t \neq Tr_{q^{2m}}(\alpha)
\end{array} \right..\\
\end{equation*}
\end{lem}
\begin{proof}
Fix $\gamma \in R_{\alpha}$. Then we get $R_{\alpha}=\gamma+\textnormal{GF}(q^{2m})$ and 
$\textnormal{GF}(q^{2m})$ is a subspace of $\textnormal{GF}(q^n)$, a  $\textnormal{GF}(2)$-linear space. 
Now define a mapping $\varphi : \textnormal{GF}(q^{2m}) \to \textnormal{GF}(q)$ as $\varphi (\beta) = St(\beta+\gamma)+St(\gamma)$.\\
By Lemma \ref{lem7}, for any $\beta_1, \beta_2 \in \textnormal{GF}(q^{2m})$, we have
\begin{align*}
\varphi(\beta_1+\beta_2) &=St(\beta_1+\beta_2+\gamma)+St(\gamma)=\\
		& = [Tr_{q^{2m}}(\beta_1+\beta_2)]^2+Tr_{q^{2m}}(\alpha(\beta_1+\beta_2))\\
		& = [Tr_{q^{2m}}(\beta_1)^2+Tr_{q^{2m}}(\alpha\beta_1)]+
			[Tr_{q^{2m}}(\beta_2)^2+Tr_{q^{2m}}(\alpha\beta_2)]\\
		& = St(\beta_1+\gamma)+St(\gamma)+St(\beta_2+\gamma)+St(\gamma)\\
		& = \varphi(\beta_1)+\varphi(\beta_2)
\end{align*}
and therefore $\varphi$ is a GF(2)-homomorphism. Furthermore
\begin{equation*}
\forall s \in \textnormal{GF}(q), \exists \beta \in \textnormal{GF}(q^{2m}), Tr_{q^{2m}}(\beta)=0, Tr_{q^{2m}}(\alpha\beta)=s,
\end{equation*} 
and so we get $St(\beta+\gamma)=St(\gamma)+s$, and $\varphi(\beta)=St(\beta+\gamma)+St(\gamma)=St(\gamma)+s+St(\gamma)=s$. Thus $\varphi$  is surjective. Homomorphism Theorem implies
\begin{equation}
\textnormal{GF}(q^{2m})/\textnormal{ker}\varphi \cong \textnormal{GF}(q). \label{eq2}
\end{equation} 
$\textnormal{GF}(q^{2m})$ is divided into cosets $\beta$+ker$\varphi$  and accordingly $R_{\alpha}=\gamma+\textnormal {GF}(q^{2m})$  is also 
divided into cosets $\gamma+(\beta+\textnormal{ker} \varphi)$. 
All the elements in each coset have the same subtrace and elements in distinct cosets have distinct subtraces.  
Moreover the number of cosets is $q$  by \eqref{eq2}. Thus for each $s \in \textnormal{GF}(q)$, 
there are $q^{2m-1}$ elements of which trace and subtrace is $(Tr_{q^{2m}}(\alpha), s)$  respectively in $R_{\alpha}$. 
\end{proof}

%
%

\begin{lem} \label{lem9}
If $\alpha \in \textnormal{GF}(q)^*$, then
\begin{equation*}
F(n, t, s, \alpha)= \left \{ 
\begin{array}{ll}
	2q^{2m-1}, & t=0, Tr_{q:2}(s/\alpha^2)=mk+1,\\
	0, & \text{otherwise}.
\end{array} \right.\\
\end{equation*}
\end{lem}
Proof. Fix arbitrary $\gamma \in R_{\alpha}$. Then we have $R_{\alpha}=\gamma+\textnormal{GF}(q^{2m})$ and 
Lemma \ref{lem7} implies that $St(\beta+\gamma)=St(\gamma)+[Tr_{q^{2m}}(\beta)]^2+\alpha \cdot Tr_{q^{2m}}(\beta)$  
for any $\beta+\gamma \in R_{\alpha}$.  Set $b=\gamma \gamma^{q^{2m}} \in \textnormal{GF}(q^{2m})$. 
Then $\gamma$  is a root of irreducible polynomial $x^2+\alpha x+b$ over $\textnormal{GF}(q^{2m})$.
Let \textbf{Irr}$(\gamma, \textnormal{GF}(q))=p(x), \textnormal{deg} ~ p(x)=l$. Then we have lcm$(2m, l)=n=4m$  and $d=n/l$  is odd. 
Thus  $l \equiv 0\pmod{4}, ~ 2m \equiv l/2\pmod{l}$ and
\begin{align*}
p(x) & = (x+\gamma)(x+\gamma^q) \cdots (x+\gamma^{q^{l/2-1}}) (x+\gamma^{q^{l/2}}) \cdots  (x+\gamma^{q^{l-1}})\\
		& = (x+\gamma)(x+\gamma^q) \cdots (x+\gamma^{q^{l/2-1}}) (x+\gamma^{q^{2m}}) \cdots  (x+\gamma^{q^{2m+l/2-1}})\\
		& = [(x+\gamma)(x+\gamma^{q^{2m}})] [(x+\gamma^q) (x+\gamma^{q^{2m+1}})] \cdots [(x+\gamma^{q^{l/2-1}})(x+\gamma^{q^{2m+l/2-1}})]\\
		& = (x^2+\alpha x+b)(x^2+\alpha x+b^q) \cdots (x^2+\alpha x+b^{q^{l/2-1}}).
\end{align*}
Therefore 
\begin{align}
Tr(p) & = (l/2) \cdot \alpha=0, ~ St(p)=\binom{l/2}{2} \alpha^2+Tr_{q^{l/2}}(b),\nonumber \\
Tr(\gamma) & = d \cdot Tr(p)=0, \label{eq3}\\
St(\gamma) & = d \cdot St(p)+\binom{d}{2} Tr(p)=[l/2 \equiv 2] \alpha^2 +Tr_{q^{l/2}}(b). \nonumber
\end{align}
For $x^2+\alpha x+b$  is irreducible over $\textnormal{GF}(q^{2m}), Tr_{q^{2m}:2}(b/\alpha^2)=1$ by 3.79 of \cite{lid},  
and  $Tr_{q:2}(Tr_{q^{l/2}}(b)/\alpha^2)=1$ since $Tr_{q^{2m}:2}(b/\alpha^2)=d \cdot Tr_{q^{l/2}:2}(b/\alpha^2)=Tr_{q^{l/2}:2}(b/\alpha^2)$. 
Therefore by \eqref{eq3} we have
\begin{equation}
Tr_{q:2}[St(\gamma)/\alpha^2]=k[l/2 \equiv 2]+1. \label{eq4}
\end{equation}
Note that since $n/l$ is odd, $[l/2 \equiv 2]=[n/2 \equiv 2]=[m \textnormal{ is odd}]$.
Let $\varphi : \textnormal{GF}(q^{2m}) \to \textnormal{GF}(q)$ be the GF(2)-homomorphism introduced in Lemma \ref{lem8}. Then \eqref{eq4} implies
\begin{equation*}
\forall \beta \in \textnormal{GF}(q^{2m}), Tr_{q:2}(\varphi(\beta)/\alpha^2)=Tr_{q:2}[(St(\beta+\gamma)+St(\gamma))/\alpha^2]=0.
\end{equation*}
Hence we have $\textnormal{Im}\varphi \subseteq W_0=\big\{ s \in \textnormal{GF}(q): Tr_{q:2}(s/\alpha^2)=0 \big \}$ . 
Furthermore for any $s \in W_0$  there exists a $\beta \in \textnormal{GF}(q^{2m})$  
such that $\varphi (\beta)=[Tr_{q^{2m}}(\beta)]^2+\alpha \cdot Tr_{q^{2m}}(\beta)=s$  from 3.79 of \cite{lid}. 
Thus $\textnormal{Im}\varphi=W_0$  is of cardinality $q/2$. 
Applying Homomorphism Theorem, we have $\textnormal{GF}(q^{2m})/\textnormal{ker} \varphi \cong W_0$  and so $|\textnormal{ker}\varphi|=\frac{q^{2m}}{q/2}=2q^{2m-1}$.
$\textnormal{GF}(q^{2m})$ is divided into $q/2$ cosets of $\textnormal{ker}\varphi$. All the elements in the same coset have the same subtrace
and the elements in distinct cosets have distinct subtraces.
The result follows from the fact mentioned above and \eqref{eq3}. $\Box$
%
%

\begin{lem} \label{lem10}
\begin{equation*}
F(n, t, s, 0)= \left \{ 
\begin{array}{ll}
	q^{2m-1}, & t=0,\\
	0, & t \neq 0.
\end{array} \right.\\
\end{equation*}
\end{lem}
\begin{proof} 
Lemma \ref{lem6} implies that $\forall \beta \in R_0, Tr(\beta)=Tr_{q^{2m}}(0)=0$. And in Lemma \ref{lem7}, putting $\gamma=0$ , we have $\forall \beta \in R_0, St(\beta)=[Tr_{q^{2m}}(\beta)]^2$. For any $s \in \textnormal{GF}(q), [Tr_{q^{2m}}(\beta)]^2=s \Leftrightarrow Tr_{q^{2m}}(\beta)=s^{q/2}$ and there are $q^{2m-1} ~ \beta$'s such that $Tr_{q^{2m}}(\beta)=s^{q/2}$  in  $\textnormal{GF}(q^{2m})$. This yields the result.     
\end{proof}

\begin{proof}[Proof of Thoerem 4]
By Lemmas  \ref{lem8}-\ref{lem10} we obtain

1) If $t=s=0$ then
\begin{align*}
F(n, 0, 0) & = F(n, 0, 0, 0)+\sum_{\alpha \in \textnormal{GF}(q)^*}F(n, 0, 0, \alpha)+\sum_{\alpha \in \textnormal{GF}(q^{2m})\setminus \textnormal{GF}(q)}F(n, 0, 0, \alpha)\\
	& = q^{2m-1}+[mk \textnormal{ is odd}]\cdot (q-1) \cdot 2q^{2m-1}+(q^{2m-1}-q)q^{2m-1} \\
	& = q^{n-2}-(-1)^{mk}(q-1)q^{2m-1} \\
	& = q^{n-2}-(-1)^{mk}v(s)q^{2m-1}.
\end{align*}

2) If $t=0, s \neq 0$ then
\begin{align*}
F(n, 0, s) & = F(n, 0, s, 0)+\sum_{\alpha \in \textnormal{GF}(q)^*}F(n, 0, s, \alpha)+\sum_{\alpha \in \textnormal{GF}(q^{2m})\setminus \textnormal{GF}(q)}F(n, 0, s, \alpha)\\
	& = q^{2m-1}+\sum_{\alpha \in \textnormal{GF}(q)^*}[Tr_{q:2}(s/\alpha^2)=mk+1] \cdot 2q^{2m-1}+(q^{2m-1}-q)q^{2m-1} \\
	& = q^{2m-1}+(q/2-[mk \textnormal{ is odd}]\cdot 2q^{2m-1}+(q^{2m-1}-q)q^{2m-1}) \\
	& = q^{n-2}+(-1)^{mk}q^{2m-1} \\
	& = q^{n-2}-(-1)^{mk}v(s)q^{2m-1}.
\end{align*}

3) If $t \neq 0$ then
\begin{align*}
F(n, t, s) & = F(n, t, s, 0)+\sum_{\alpha \in \textnormal{GF}(q)^*}F(n, t, s, \alpha)+\sum_{\alpha \in \textnormal{GF}(q^{2m})\setminus \textnormal{GF}(q)}F(n, t, s, \alpha)\\
	& = 0+0+ \sum_{\alpha \in \textnormal{GF}(q^{2m})\setminus \textnormal{GF}(q)}[Tr_{q^{2m}}(\alpha)=t]q^{2m-1}\\
	& = q^{2m-1}q^{2m-1}=q^{n-2}.
\end{align*}
We have completed the proof. 
\end{proof} 

We give an example below. Let $n=3, q=4$ and $\alpha$ be a root of $x^2+x+1$. Then $GF(4) = \{0, 1, \alpha, \alpha^2\}$.
From Theorem \ref{thr2}, we can calculate $F(3, t, s), t, s \in \textnormal{GF(4)}$ as in Table 1.\\

\centerline{
\begin{tabular}{c   c   c   c   c} 
\hline
& & $\quad ~~ s $ &  &  \\  \cline{2-5}
$t \quad$ & $0 \quad$ & $1 \quad$ & $\alpha \quad$ & $\alpha^2 $ \\	
\hline
$0 \quad$ & $7 \quad$ & $3 \quad$ & $3 \quad$ & $3 $ \\
$1 \quad$ & $3 \quad$ & $7 \quad$ & $3 \quad$ & $3 $ \\
$\alpha \quad$ & $3 \quad$ & $3 \quad$ & $3 \quad$ & $7 $ \\
$\alpha^2 \quad$ & $3 \quad$ & $3 \quad$ & $7 \quad$ & $3$ \\
\hline
\end{tabular}
}
\vspace{0.5cm}
\centerline{\textbf{Table 1.} The values of $F(3, t, s)$ for $t, s \in \textnormal{GF}(4)$}
\vspace{0.3cm}

Applying Theorem \ref{thr1}, we get Table 2.

\vspace{0.5cm}

\centerline{
\begin{tabular}{c   c   c   c   c}
\hline
& & $\quad ~~ s $ &  &  \\  \cline{2-5}
$t \quad$ & $0 \quad$ & $1 \quad$ & $\alpha \quad$ & $\alpha^2 $ \\	
\hline
$0 \quad$ & $2 \quad$ & $1 \quad$ & $1 \quad$ & $1 $ \\
$1 \quad$ & $1 \quad$ & $2 \quad$ & $1 \quad$ & $1 $ \\
$\alpha \quad$ & $1 \quad$ & $1 \quad$ & $1 \quad$ & $2 $ \\
$\alpha^2 \quad$ & $1 \quad$ & $1 \quad$ & $2 \quad$ & $1$ \\
\hline
\end{tabular}
}
\vspace{0.5cm}
\centerline{\textbf{Table 2.} The values of $P(3, t, s)$ for $t, s \in \textnormal{GF}(4)$} 
\vspace{0.5cm}

{\bf Acknowledgement}. We would like to thank anonymous referees for their valuable comments and suggestions. 





 \end{document}